\documentclass[11pt,a4paper,oneside,reqno]{amsart}

\usepackage[all]{xy}
\SelectTips{cm}{}  \calclayout
\usepackage{amsfonts,amssymb,amscd,amsmath,enumerate,verbatim,calc}

\usepackage{amscd,amssymb,amsopn,amsmath,amsthm,graphics,amsfonts,enumerate,verbatim,calc}
\usepackage[dvips]{graphicx}
\usepackage[colorlinks=true,linkcolor=blue,citecolor=blue]{hyperref}
\input xy
\xyoption{all}

\newcommand{\rt}{\rightarrow}
\newcommand{\lrt}{\longrightarrow}

\newcommand{\st}{\stackrel}
\newcommand{\pa}{\partial}

\newcommand{\C}{\mathbf{C} }
\newcommand{\D}{\mathbb{D} }
\newcommand{\K}{\mathbb{K} }

\newcommand{\I}{\mathbf{I} }

\newcommand{\F}{\mathbf{F} }
\newcommand{\G}{\mathbf{G} }

\newcommand{\CM}{\mathcal{M}}
\newcommand{\CO}{\mathcal{O}}

\newcommand{\Z}{\mathbb{Z} }

\newcommand{\CA}{\mathcal{A} }

\newcommand{\CC}{\mathcal{C} }
\newcommand{\DD}{\mathcal{D} }
\newcommand{\CF}{\mathcal{F} }
\newcommand{\CG}{\mathcal{G} }
\newcommand{\CL}{\mathcal{L} }
\newcommand{\CH}{\mathcal{H} }

\newcommand{\CI}{\mathcal{I} }
\newcommand{\CK}{\mathcal{K} }
\newcommand{\CP}{\mathcal{P} }
\newcommand{\CQ}{\mathcal{Q} }

\newcommand{\CS}{\mathcal{S} }
\newcommand{\CT}{\mathcal{T} }

\newcommand{\CJ}{\mathcal{J} }

\newcommand{\Inj}{{\rm{Inj}}}
\newcommand{\Prj}{{\rm{Proj}}}
\newcommand{\Flat}{{\rm{Flat}}}

\newcommand{\KCOFR}{{\K({\rm{Cof}} R)}}

\newcommand{\KFX}{{\mathbb{K}({\Flat} X)}}

\newcommand{\KCOFX}{{\K({\rm{Cof}} X)}}

\newcommand{\KPIFX}{{\K({\rm{Pinf}} X)}}
\newcommand{\KPIFR}{{\K({\rm{Pinf}} R)}}

\newcommand{\KPFX}{{\mathbb{K}_{\rm{pac}}({\Flat} X)}}

\newcommand{\DPFX}{{\mathbb{D}_{\rm{pur}}({\Flat} X)}}

\newcommand{\KPIX}{{\mathbb{K}_{\rm{pac}}({\Inj} X)}}

\newcommand{\KPR}{{\mathbb{K}({\Prj}  R)}}

\newcommand{\KFR}{{\mathbb{K}({\Flat}  R)}}
\newcommand{\KIX}{{\mathbb{K}({\Inj}  X)}}
\newcommand{\KIR}{{\mathbb{K}({\Inj}  R)}}

\newcommand{\DPABX}{{\mathbb{D}_{\rm{pur}}({\mathrm{Abs}} X)}}

\newcommand{\KABX}{{\mathbb{K}({\mathrm{Abs}} X)}}
\newcommand{\KPABX}{{\mathbb{K}_{\rm{pac}}({\mathrm{Abs}} X)}}

\newcommand{\Ker}{{\rm{Ker}}}

\newcommand{\Hom}{{\rm{Hom}}}
\newcommand{\Ext}{{\rm{Ext}}}

\newcommand{\Lim}{\underset{\underset{n \in \mathbb{N}}{\longrightarrow}}{\lim}}

\newtheorem{theorem}{Theorem}[section]
\newtheorem{corollary}[theorem]{Corollary}
\newtheorem{lemma}[theorem]{Lemma}
\newtheorem{proposition}[theorem]{Proposition}

\theoremstyle{definition}

\newtheorem{remark}[theorem]{Remark}

\theoremstyle{plain}

\theoremstyle{definition}

\numberwithin{equation}{section}
\begin{document}

\title[The homotopy category of pure injective flats]{The homotopy category of pure injective flats  and  Grothendieck duality}
\author[ Esmaeil  Hosseini ]{Esmaeil  Hosseini}

\address{Department of Mathematics, Shahid Chamran University of Ahvaz,
 P.O.Box: 61357-83151, Ahvaz, Iran.} \email{e.hosseini@scu.ac.ir}

%\dedicatory{Dedicated to Edgar E. Enochs on his 90th birthday}
\begin{abstract}
Let $(X,\CO_X)$ be a locally noetherian scheme with a dualizing
complex $\DD$. We prove that  $\DD\otimes^\bullet_{\CO_X}
\texttt{-}:\KPIFX\lrt\KIX$ is an equivalence of triangulated
categories where $\KIX$ is  the homotopy category of injective
quasi-coherent $\CO_X$-modules and $\KPIFX$ is the homotopy category
of pure injective flat quasi-coherent $\CO_X$-modules. Where  $X$ is
affine, we show that this equivalence is the infinite completion of
the Grothendieck duality theorem. Furthermore, we prove that
$\DD\otimes^\bullet_{\CO_X} \texttt{-}$ induces an equivalence
between the pure derived category of flats and the pure derived
category of absolutely pure quasi-coherent $\CO_X$-modules.
\end{abstract}
\maketitle
\section{Introduction}

Let $X$ be a locally noetherian scheme. The bounded derived category
$\mathbb{D}^b_{\mathrm{coh}}(\mathfrak{Qco}X)$ of coherent sheaves
of $\CO_X$-modules is a natural object of study in modern algebraic
geometry (\cite{BO02, Bir06, Orl10}). It's fundamental property is
the Grothendieck duality which describes a bounded complex $\DD$ of
injective quasi-coherent $\CO_X$-modules such that
\begin{align}\label{d3}
\xymatrix{\mathbf{Hom}^\bullet_{\CO_X}(\texttt{-}, \DD):
{\mathbb{D}^b}_{\mathrm{coh}}(\mathfrak{Qco}X)^{\mathrm{op}}\ar[r]&{\mathbb{D}^b}_{\mathrm{coh}}(\mathfrak{Qco}X)}
\end{align}
is an equivalence of triangulated categories (see \cite[$\S$
V]{Har66}, \cite{Lip}, \cite{Ne96, Ne10c, Ne10a}, \cite{Ne08},
\cite{Ne11} and \cite[$\S$ 3]{Co}).

In the case  $X=\mathrm{Spec}R$ ($R$ is a commutative noetherian
ring), the Grothendieck duality tells us that

\begin{align}\label{d1}
\xymatrix{ \mathbf{Hom}^\bullet_R(\texttt{-},
\DD(X)):\mathbb{D}^{b}(\mathrm{mod}
R)^{\mathrm{op}}\ar[r]&\mathbb{D}^{\mathrm{b}}(\mathrm{mod} R)}
\end{align}is an equivalence of triangulated categories whenever
$\mathbb{D}^{\mathrm{b}}(\mathrm{mod} R)$ is the bounded derived
category of finitely generated $R$-modules (see
\cite[§II.10]{Har66}, \cite[Lemma 3.1.4]{Co}). Furthermore, if
$\KPR$ (resp $\KIR$) is the homotopy category of projective (resp.
injective) $R$-modules, then the existence of $\DD(X)$ induces the
equivalence

\begin{align}\label{d20}
  \xymatrix@C+6pc@R+2pc{\KPR \ar[r]^{\DD(X)\otimes_{R}^\bullet-} &
 \KIR  }
\end{align}
of triangulated categories where its restriction to compact objects
is \eqref{d1} (\cite[Theorem 4.2]{IK}). In fact, \eqref{d20} is the
infinite completion of \eqref{d1} that is based on three factors,
$\KPR$, $\KIR$ and the Grothendieck duality between them.
Unfortunately, in non-affine cases, there is a gaping hole on the
projective  side of \eqref{d20}, because the homotopy  category of
projective objects is trivial in some situations (see \cite[$\S$
8]{M07}, \cite{Ne08}, \cite{Ne10} and \cite{Ne11}). In \cite{Ho17},
when $X$ is a locally noetherian scheme with a dualizing complex
(not necessarily \textbf{semi-separated}), we proved that the
homotopy category $\KCOFX$ of cotorsion flat quasi-coherent sheaves
of $\CO_X$-modules is the best replacement for the homotopy category
of projectives. This is motivated us to ask the following question.

\vspace{.5cm}

$\textbf{Question:}$ Does \eqref{d20} hold for $\KCOFR$ instead of
$\mathbf{K}(\mathrm{Proj}R)$?

\vspace{.5cm} The aim of this work is to find an answer to this
question. Indeed, if $X$ is a locally noetherian scheme with a
dualizing complex $\DD$, then,  we prove the following assertions.
\begin{itemize}
\item [(i)] $\KCOFX=\KPIFX$ where $\KPIFX$ is the homotopy category of pure injective flat quasi-coherent
$\CO_X$-modules.
\item [(ii)] there exists an equivalence \begin{align}\label{d0021} \DD\otimes^\bullet_{\CO_X}
\texttt{-}:\KPIFX\lrt\KIX
\end{align} of triangulated categories where $\KIX$ is the homotopy category of injective
quasi-coherent $\CO_X$-modules.
\item [(iii)] absolutely pure quasi-coherent $\CO_X$-modules are injective.
\item [(iv)] there exists an equivalence \begin{align}
\xymatrix@C+3pc@R+3pc{\DPFX
\ar@<.75ex>[r]^{\DD\otimes^\bullet_{\CO_X}-}&
 \DPABX \ar@<0.5ex>[l]^{\mathbf{Hom}_{\mathrm{qc}}^\bullet(\DD,-)} }
\end{align}
of triangulated categories where $\DPFX$ is the pure derived
category of flats and $\DPABX$ is the pure derived category of
absolutely pure $\CO_X$-modules.
\end{itemize}

\section{Preliminaries}
In this section we recall some basic definitions and results, needed
throughout the article. Throughout the paper, $(X,\CO_X)$ will
always denote a \textbf{locally} \textbf{noetherian} scheme with the
structure sheaf $\CO_X$, $R$ is a commutative noetherian ring with
identity, $\CO_X$-modules are quasi-coherent sheaves  of
$\CO_X$-modules, sheaves are sheaves of $\CO_X$-modules and modules
are $R$-modules unless otherwise specified. The category of all
$\CO_X$-modules (resp. sheaves, $R$-modules) is denoted by
$\mathfrak{Qco}X$ (rep. $\mathfrak{Mod}X$, $R$-Mod). The stalk of a
given $\CO_X$-module (resp. sheaf) $\CF$ at a point $x\in X$ is
denoted by $\CF_x$. Particularly, the stalk of $\CO_X$ at $x\in X$
is denoted by $\CO_x$. Note that, $\CO_x$ is a commutative
noetherian local ring and $\CF_x$ is  an $\CO_x$-module. The
category of $\CO_x$-modules is denoted by $\CO_x$-Mod. For more
backgrounds on sheaves and schemes  see \cite[\S. II]{Har97}.

We shall always assume that  $(\CA, \otimes)$ is a symmetric closed
monoidal Grothendieck category and $\C(\CA)$ is the category of all
complexes (complexes are written cohomologically) in $\CA$. For each
pair $\F=(\CF^n,\pa^n_{\F})$, $\G=(\CG^n,\pa^n_{\G})$ of complexes
in $\CA$, the \textit{Hom} \textit{complex} of $\F$, $\G$ is defined
by
$$\mathbf{Hom}^\bullet_{\CA}(\mathbf{F},\G)=({\prod_{i\in\Z}}{{\CH}om}_{\CA}(\CF^{i},\CG^{i+n}),\sigma^n)_{n\in\Z}$$
such that
$$\sigma^n ((f^i)_{i\in\Z})=\partial_\G^{i+n} f^i-(-1)^n f^{i+1}(\partial_\F^{i})$$
where ${\CH}om_{\CA}(\textmd{-},\textmd{-})$ is the internal hom in
$\CA$. The \textit{tensor} \textit{product} of $\F$, $\G$ is defined
by  the complex $\F\otimes_{\CO_X}^\bullet\G$  whose component in
degree $n$ is
$$(\F\otimes_{\CO_X}^\bullet\G)^
n:=\bigoplus_{i\in\Z}\CF^i\otimes_{\CO_X} \CG^{n-i}$$ and whose
differential is given by
$$\partial_{\F\otimes^\bullet \G}^n|_{\CF^i\otimes \CG^{n-i}}=(\partial_\F^i\otimes id_\G)+(-id_\F)^{i}\otimes\partial_\G^{n-i}).$$

A \textit{double} \textit{complex} in  $\CA$ is a family
$\C=\{\CC^{m,n}\}$ of objects in $\CA$, together with maps
$\pa^{h,m}:\CC^{m,n}\lrt\CC^{m+1,n}$ (horizontal) and
$\pa^{v,n}:\CC^{m,n}\lrt\CC^{m,n+1}$ (vertical) such that
$$\pa^{h,m+1}\pa^{h,m}=\pa^{v,n+1}\pa^{v,n}=\pa^{h,m}\pa^{v,n-1}+\pa^{v,n}\pa^{h,m-1}=0.$$
The \textit{total} \textit{complex} of $\C$ is defined by the
complex
$\mathbf{Tot}(\C)=\{\prod_{m+n=k}^{\mathrm{qc}}\CC^{m,n},\pa^k\}$
where $\pa^k=\pa^{h}+\pa^{v}$. Let
$\mathbf{P}=(\CP^n,\pa^n_{\mathbf{P}})$ and $\I=(\CI^n,\pa^n_{\I})$
be complexes of objects in  $\CA$. If we convert $\mathbf{P}$ into a
chain complex $\G$ (write homologically) with $\CG_i=\CP^{-i}$ and
form the double complex
$\mathbf{Hom}(\G,\I)=\{{\CH}om_{\CA}(\CG_i,\CI^j)\}_{i,j\in\Z}$,
then
$\mathbf{Hom}^\bullet_{\CA}(\mathbf{P},\I)=\mathbf{Tot}(\mathbf{Hom}(\G,\I))$
(See \cite[2.7.4, pp. 62]{We}).

\textbf{Acyclic Assembly Lemma:} Let $\C$ be a double complex in
$R$-Mod. Then, Tot$(\C)$ is an acyclic complex, assuming either of
the following:
\begin{itemize}

\item [(i)] $\C$ is an upper half-plane complex with exact columns.
\item [(ii)] $\C$ is a right half-plane complex with exact rows.

\end{itemize}
\begin{proof}
\cite[Lemma 2.7.3]{We}.
\end{proof}

The derived category $\mathbb{D}(\CA)$ of $\CA$ is obtained from
$\C(\CA)$ by formally inverting all homology isomorphisms. The
homotopy category $\K(\CA)$ of $\CA$ has the same objects as in
$\C(\CA)$ and morphisms are the homotopy classes of morphisms of
complexes.

A short exact sequence $\CL$ in $\CA$ is called \textit{pure} if for
any object $\CH$, $\CL\otimes\CH$ remains exact.  A subobject $\CF$
of a given object $\CG $ is called \textit{pure} if the canonical
sequence
$\xymatrix@C-0.7pc{0\ar[r]&\CF\ar[r]&\CG\ar[r]&\CG/\CF\ar[r]&0 }$ is
pure. An object $\CC$ in $\CA$ is pure injective if it is injective
with respect to pure exact sequences. Let $\CJ$ be an injective
cogenerator of $\CA$ and $()^+={\CH}om_\CA(-,\CJ)$. By \cite[Theorem
2.6]{HZ18}, for any object $\CF$, we have a pure monomorphism
$\lambda_{\CF}:\CF\lrt\CF^{++}$ where $\CF^{++}$ is pure injective.
Particularly, $\CF^+$ is pure injective.

An object $\CF$ of $\CA$ is said to be \textit{flat} if the functor
$\CF\otimes$-  preserves exact sequences in  $\CA$. The next result
is devoted to a characterization of flat objects.

\begin{proposition}\label{elh20}
Let $\CF$ be an object of $\CA$. Then, the following conditions are
equivalent.

\begin{itemize}
\item[(i)]  $\CF$ is flat.
\item[(ii)]  $\CF^+$ is injective.
\item[(iii)] Any short exact sequence
in $\CA$ ending in $\CF$ is pure.
\end{itemize}

\end{proposition}
\begin{proof}
\cite[Theorem 2.3]{HZ18}.
\end{proof}

\begin{lemma}\label{mainf670011}
If two-out of three objects in a short pure exact sequence in $\CA$
is flat then so is the third.
\end{lemma}
\begin{proof}
Assume that\begin{align}\label{Reza khan mir q3}
\xymatrix@C-0.7pc@R-0.9pc{0 \ar[r]&\CF\ar[r]&\CG \ar[r]&\CK\ar[r]&0}
\end{align}is a pure exact sequence in $\CA$.  By assumption, we
have $\CG^{+}=\CF^+\oplus\CK^+$. Now, we apply Proposition
\ref{elh20}.

\end{proof}

An object $\CC$ in $\CA$ is said to be \textit{cotorsion} if for any
flat object $\CF$, $\Ext^1_{\CA}(\CF,\CC)=0$. It is called
\textit{cotorsion} (resp. \textit{pure} \textit{injective})
\textit{flat} if it is both cotorsion (resp. pure injective) and
flat. In the case $\CA=\mathfrak{Qco}X$, any $\CO_X$-module $\CF$
can be embedded in a cotorsion $\CO_X$-module $\CC$ such that
$\CC/\CF$ is a flat $\CO_X$-module (\cite{EE05}). So, by Lemma
\ref{mainf670011}, any flat $\CO_X$-module can be purely embedded in
a cotorsion flat $\CO_X$-module.

It is known that
$(\mathfrak{Mod}X,\textmd{-}\otimes_{\CO_X}\textmd{-})$ is a
well-known example of a symmetric closed monoidal Grothendieck
category. In general case, the internal hom
${\mathcal{H}}om_{\CO_X}(\textmd{-},\textmd{-})$ in
$\mathfrak{Mod}X$ does not preserve  $\CO_X$-modules, i.e. for each
pair $\CF$, $\CG$ of $\CO_X$-modules,
${\mathcal{H}}om_{\CO_X}(\CF,\CG)$ need not be an $\CO_X$-module. To
remove this difficulty, we use the right adjoint of the inclusion
functor $i_X:\mathfrak{Qco}X\lrt\mathfrak{Mod}X$. By the Special
Adjoint Functor Theorem $i_X$ admits a right adjoint
$Q:\mathfrak{Mod}X\lrt\mathfrak{Qco}X$ which is called the
\textit{coherator} functor. We use this functor and define the
\textit{internal} \textit{hom} in $\mathfrak{Qco}X$ by
${\mathcal{H}}om_{\rm{qc}}(\CF,\CG)=Q({\mathcal{H}}om_{\CO_X}(\CF,\CG))$.
This turns $(\mathfrak{Qco}X, \textmd{-}\otimes_{\CO_X}\textmd{-})$
to a symmetric monoidal closed Grothendieck category, i.e.  for any
$\CO_X$-module $\CF$, ${\mathcal{H}}om_{\rm{qc}}(\CF,-)\ :
\mathfrak{Qco}X \rightarrow \mathfrak{Qco}X$ is the right adjoint of
$ -\otimes_{\CO_X} \CF \  : \mathfrak{Qco}X \rightarrow
\mathfrak{Qco}X$. Notice that if $\CF$ is a coherent $\CO_X$-module,
then,
${\mathcal{H}}om_{\rm{qc}}(\CF,-)\cong{\mathcal{H}}om_{\CO_X}(\CF,-)$.
Moreover, $Q$ helps us to define the quasi-coherent product of a
given family $\{\CG_i\}_{i\in I}$ of objects in $\mathfrak{Qco}X$,
i.e. the quasi-coherent product of $\{\CG_i\}_{i\in I}$  is defined
by $\prod_{i\in I}^{\mathrm{qc}}\CG_i=Q(\prod_{i\in I}\CG_i)$ where
$\prod_{i\in I}\CG_i$  is the  product of $\{\CG_i\}_{i\in I}$ in
$\mathfrak{Mod}X$.

\section{Pure injective flat $\CO_X$-modules}
The main result of this section shows that,  for each pair $\CI,\CK$
of injective $\CO_X$-modules, $\mathcal{H}om_{\mathrm{qc}}(\CI,\CK)$
is a pure injective flat $\CO_X$-module. This implies that any
cotorsion flat $\CO_X$-modules is pure injective. First, we begin by
recalling some basic properties of injective $\CO_X$-modules which
can be found in \cite{Har66} and \cite{Co}.
\begin{proposition}\label{mes1}
Let $\CI$ be an $\CO_X$-module. Then, the following conditions are
equivalent:
\begin{itemize}
\item[i)] $\CI$ is injective.
\item [ii)] For any affine open subset $U$ of $X$, $\CI|_U$ is an injective $\CO_U$-module.
\item [iii)] For any  $x\in X$, $\CI_x$ is an injective $\CO_{x}$-module.
\item [iv)] For any affine open subset $U$ of $X$, $\CI(U)$ is an injective $\CO_X(U)$-module.
\end{itemize}

\end{proposition}

\begin{proof}
\cite[Proposition II.7.17]{Har66}.
\end{proof}

\begin{proposition}\label{mes12020}
Any injective $\CO_X$-module is an injective object of
$\mathfrak{Mod}X$.
\end{proposition}
\begin{proof}
\cite[Lemma 2.1.3]{Co}.
\end{proof}
\begin{remark}\label{123456}
Let $A$ be a commutative ring and $M$ be an $A$-module. The
\textit{sheaf} \textit{associated} to $M$ on Spec$A$ is denoted by
$\widetilde{M}$ (see \cite[\S. II,  pp. 110]{Har97}). In the case
that $M$ is injective \cite[Corollary 7.14]{Har66} shows the
injectivity of $\widetilde{M}$. For $x\in X$, let
$\CJ(x)=(i_x)_*(\widetilde{J(x)})$, where
$i_x:\mathrm{Spec}\CO_x\lrt X$ be the natural inclusion, $(i_x)_*$
be the direct image functor  and $\widetilde{J(x)}$ be the
quasi-coherent sheaf on Spec$\CO_x$ associated to the injective hull
$J(x)$ of $k(x)=\CO_x/\mathfrak{m}_x$ over $\CO_x$. By
\cite[Proposition 7.6]{Har66}, $\CJ(x)$ is an indecomposable
injective object in $\mathfrak{Qco}X$ and by \cite[Proposition
II.7.5]{Har66}, it has support at the closed point $x$ of
Spec$\CO_x$. Moreover, for any sheaf of $\CO_X$-module $\CF$ (not
necessarily quasi-coherent), we have
\begin{align}
\mathrm{Hom}_{\CO_X}(\CF,(i_x)_*(\widetilde{J(x)}))\cong\mathrm{Hom}_{\mathrm{Spec}\CO_x}((i_x)^*\CF,\widetilde{J(x)})\cong
\mathrm{Hom}_{\CO_x}(\CF_x,J(x)).
\end{align}
\end{remark}

\begin{proposition}\label{mes101100985624}
For any set of cardinals $\{\aleph_x\}_{x\in X}$,  the direct sum
\begin{align}\label{mes1000985624}
\bigoplus_{x\in X}\CJ(x)^{\oplus\aleph_x}
\end{align} is an injective $\CO_X$-module, where $\CJ(x)^{\oplus\aleph_x}$ is the  direct sum of copies of $\CJ(x)$
indexed by the cardinal $\aleph_x$; moreover, any  injective
$\CO_X$-module can be written in the form \eqref{mes1000985624} with
unique cardinals $\aleph_x$.
\end{proposition}
\begin{proof}
\cite[Lemma 2.1.5]{Co} (see also, \cite[II, 7.13, 7.17]{Har66}).
\end{proof}

After this introduction,  we will discuss the structure of
quasi-coherent product. Recall that, for a commutative ring $A$, the
category of all quasi-coherent sheaves over Spec$A$ admits products,
i.e. the product of any family $\{\CF_i\}_{i\in S}$ of
quasi-coherent sheaves over Spec$A$ is $\prod_{i\in S}
^{\mathrm{qc}}\CF_i\cong \widetilde{K}$ where $K=\prod_{i\in
S}\CF_i(\mathrm{Spec}A)$. But, it's not that easy in non-affine
cases. In the following Lemma, we look at the subject in a special
case.

\begin{lemma}\label{mes1000}
For any $x\in X$, assume that $M_x$ is an $\CO_x$-module. Then, for
any affine open subset $U=\mathrm{Spec}(R)$ of $X$, $(\prod_{x\in X}
^{\mathrm{qc}} (i_x)_* \widetilde{M_x})|_U\cong\widetilde{N}$  where
$N=\prod_{x\in U} M_x$ is an $R$-module.
\end{lemma}
\begin{proof}
Let $U=\mathrm{Spec}(R)$ be an affine open subset of $X$ and
$\psi:U\lrt X$ be the inclusion. For any $x\in U$, let
$j_x:\mathrm{Spec}\CO_x\lrt U$ be the inclusion and $(j_x)_*$ be the
direct image functor. Then, $i_x=\psi j_x$ and so $(i_x)_*=\psi_*
(j_x)_*$ and $(i_x)_*|_U=(j_x)_*$. We know that $\mathfrak{Qco}U$
has products, i.e. the product of the family
$\{{j_x}_*\widetilde{M_x}\}_{x\in U}$ is $\widetilde{N}$. Let  $\CM$
be an object of $\mathfrak{Qco}U$ with a collection of morphisms
$\{s_x:\CM\lrt {j_x}_* \widetilde{M_x}\}_{x\in U}$. For any $x\in
X$, we have a morphism $\psi_x:\psi_*\CM\lrt (i_x)_*\widetilde{M_x}$
in $\mathfrak{Qco}X$ (for any $x\in X-U$, $\psi_x=0$). By the
assumption there exists a unique morphism $\varphi:\psi_*\CM\lrt
(\prod_{x\in X} ^{\mathrm{qc}} (i_x)_* \widetilde{M_x})$ such that
$\pi_x\varphi=\psi_x$. The restriction on $U$ implies the following
commutative diagram
\begin{align}\label{Reza khan}
\xymatrix@C2pc@R3pc{(\prod_{x\in X} ^{\mathrm{qc}}
(i_x)_* \widetilde{M_x})|_U\ar[rr]^{(\pi_x)|_U}&&(j_x)_*\widetilde{M_x}.&\\
&\CM\ar[ul]_{\varphi|_U}\ar[ru]^{s_x}&}
\end{align}This shows that $(\prod_{x\in X} ^{\mathrm{qc}} (i_x)_* \widetilde{M_x})|_U$ is the
product of $\{{j_x}_*\widetilde{M_x}\}_{x\in U}$ in
$\mathfrak{Qco}U$. The uniqueness of product in $\mathfrak{Qco}U$
implies the isomorphism $(\prod_{x\in X} ^{\mathrm{qc}} (i_x)_*
\widetilde{M_x})|_U\cong\widetilde{N}$.
\end{proof}

Now we are able to prove the main result of this section,  that for
each pair $\CI$, $\CK$ of injective $\CO_X$-modules,
$\mathcal{H}om_{\mathrm{qc}}(\CI,\CK)$ is a flat $\CO_X$-module. If
$X=\mathrm{Spec}R$, then for each pair $I$, $J$ of injective
$R$-modules, we have the isomorphism
$$\mathcal{H}om_{\mathrm{qc}}(\widetilde{I},\widetilde{J})\cong \widetilde{\mathrm{Hom}_R(I,J)}.$$
This shows that
$\mathcal{H}om_{\mathrm{qc}}(\widetilde{I},\widetilde{J})$ is flat
(see  \cite{TT}).

\begin{theorem}\label{mainf}
Let $\CI$, $\CK$  be a pair of  injective objects in
$\mathfrak{Qco}X$. Then, $\mathcal{H}om_{\mathrm{qc}}(\CI,\CK)$ is a
 flat $\CO_X$-module.
\end{theorem}
\begin{proof}
By Propositions \ref{mes12020} and \ref{mes101100985624}, we have a
split monomorphism $\CK\lrt\prod_{x\in X}\prod_{\aleph_x}\CJ(x)$ in
$\mathfrak{Mod}X$. So, $\CK\lrt\prod^{\mathrm{qc}}_{x\in
X}\prod^{\mathrm{qc}}_{\aleph_x}\CJ(x)$ is a split monomorphism of
injectives in $\mathfrak{Qco}X$ and hence
$\mathcal{H}om_{\mathrm{qc}}(\CI,\CK)$ is a direct summand of
$\prod^{\mathrm{qc}}_{x\in
X}\prod^{\mathrm{qc}}_{\aleph_x}\mathcal{H}om_{\mathrm{qc}}(\CI,(i_x)_*\widetilde{J(x)}).$
So, for any affine open subset $U=\mathrm{Spec}R$ of $X$, we have
the following isomorphism
$$\{\prod^{\mathrm{qc}}_{x\in
X}\prod^{\mathrm{qc}}_{\aleph_x}\mathcal{H}om_{\mathrm{qc}}(\CI,(i_x)_*\widetilde{J(x)})\}|_U\cong\{\prod^{\mathrm{qc}}_{x\in
U}\prod^{\mathrm{qc}}_{\aleph_x}(i_x)_*\mathcal{H}om_{\mathrm{qc}}((i_x)^*\CI,\widetilde{J(x)})\}|_U.$$Then,
it is enough to prove the flatness of
$$(\prod^{\mathrm{qc}}_{x\in
X}\prod^{\mathrm{qc}}_{\aleph_x}(i_x)_*\mathcal{H}om_{\mathrm{qc}}((i_x)^*\CI,\widetilde{J(x)}))|_U$$
in $\mathfrak{Qco}U$ where $U=\mathrm{Spec}R$ is an affine open
subset of $X$. But by Lemma \ref{mes1000} and Remark \ref{123456},
$(\prod^{\mathrm{qc}}_{x\in
X}\prod^{\mathrm{qc}}_{\aleph_x}(i_x)_*\mathcal{H}om_{\mathrm{qc}}(\widetilde{\CI_x},\widetilde{J(x)}))|_U\cong
\widetilde{I}$ where $I=\prod_{x\in
U}\prod_{\aleph_x}\Hom_{\CO_x}(\CI_x,\CK_x)$. Since for any $x\in
U$, $R_x=\CO_x$ is a noetherian ring, then,
$\Hom_{\CO_x}(\CI_x,\CK_x)$ is a flat $R_x$-module and so it is a
flat $R$-module ($R_x$ is a flat $R$-module).  Consequently, $I$ is
a flat $R$-module ($R$ is noetherian). This shows that
$\prod^{\mathrm{qc}}_{x\in
X}\prod^{\mathrm{qc}}_{\aleph_x}\mathcal{H}om_\mathrm{qc}(\CI,(i_x)_*\widetilde{J(x)})$
is flat and so we are done.
\end{proof}

\begin{theorem}\label{hom01325871}
Let $\CI$ be an injective object in $\mathfrak{Qco}X$. Then
${\mathcal{H}}om_{\rm{qc}}(\textmd{-},\CI):\mathfrak{Qco}X
\rightarrow \mathfrak{Qco}X$ is an exact functor.
\end{theorem}
\begin{proof}
By the same method as used in the proof of Theorem \ref{mainf}, for
a given exact sequence  $\CL$ of $\CO_X$-modules,
$\mathcal{H}om_{\mathrm{qc}}(\CL,\CI)$ is a direct summand of
$$\prod^{\mathrm{qc}}_{x\in
X}\prod^{\mathrm{qc}}_{\aleph_x}(i_x)_*\mathcal{H}om_{\mathrm{qc}}((i_x)^*\CL,\widetilde{J(x)}).$$
Then, it is enough to prove that for an arbitrary affine open subset
$U=\mathrm{Spec}R$ of $X$, $$(\prod^{\mathrm{qc}}_{x\in
X}\prod^{\mathrm{qc}}_{\aleph_x}(i_x)_*\mathcal{H}om_\mathrm{qc}((i_x)^*\CL,\widetilde{J(x)}))|_U$$
is an exact sequence of $\CO_U$-modules. But, by Lemma \ref{mes1000}
and Remark \ref{123456}, $$(\prod^{\mathrm{qc}}_{x\in
X}\prod^{\mathrm{qc}}_{\aleph_x}(i_x)_*\mathcal{H}om_{\mathrm{qc}}((i_x)^*\CL,\widetilde{J(x)}))|_U\cong
\widetilde{I}$$ where $I=\prod_{x\in
U}\prod_{\aleph_x}\Hom_{\CO_x}(\CL_x,\CI_x)$ is an exact sequence of
$R$-modules. We conclude that
${\mathcal{H}}om_{\rm{qc}}(\textmd{-},\CI)$ is locally exact and
hence it is exact.
\end{proof}

\begin{proposition}\label{z440048971}
Let $\{\CG_i\}_{i\in I}$ be a family  of pure injective
$\CO_X$-modules. Then, $\prod_{i\in I}^{\mathrm{qc}}\CG_i$ is a pure
injective $\CO_X$-module.
\end{proposition}
\begin{proof}
Let $\CL$ be a pure short exact sequence of $\CO_X$-modules. Then,
by the adjoint property of $(i_X,Q)$ and the following isomorphisms
$$\mathrm{Hom}_{\CO_X}(\CL,\prod_{i\in I}^{\mathrm{qc}}\CG_i)\cong
\mathrm{Hom}_{\CO_X}(\CL,\prod_{i\in I}\CG_i)\cong\prod_{i\in
I}\mathrm{Hom}_{\CO_X}(\CL,\CG_i)$$we are done.

\end{proof}

\begin{proposition}\label{mainf6700}
A flat $\CO_X$-module is cotorsion if and only if it  is pure
injective.
\end{proposition}
\begin{proof}
Let $\CF$ be a cotorsion flat $\CO_X$-module. So, by Theorem
\ref{mainf}, Lemma \ref{mainf670011} and the following pure exact
sequence
\begin{align}\label{Reza khan mir}
\xymatrix@C-0.7pc@R-0.9pc{ 0
\ar[r]&\CF\ar[r]^{\lambda_{\CF}}&\CF^{++}
\ar[r]&\CF^{++}/\CF\ar[r]&0,}
\end{align}of $\CO_X$-modules
we conclude that  $\CF^{++}/\CF$ is flat. Consequently, \eqref{Reza
khan mir} splits and hence $\CF$ is pure injective. The converse is
trivial.
\end{proof}

\section{Infinite completion of Grothendieck duality theorem}
Let  $\KFX$ (resp. $\KIX$, $\KPIFX$, $\KCOFX$) be the triangulated
subcategories of $\K(X)=\K(\mathfrak{Qco}X)$ consisting of complexes
of flat (resp. injective, pure injective flat and cotorsion flat)
$\CO_X$-modules. In the present  section, we prove the equivalence
$\DD\otimes^\bullet_{\CO_X} \texttt{-}:\KPIFX\lrt\KIX$ of
triangulated subcategories of $\KFX$. The category $\KFX$ was first
studied by Neeman (\cite{Ne08, Ne10, Ne11b}) who discovered a new
perspective of the Grothendieck duality theorem (\cite{Ne11},
\cite{Ne10a, Ne10c}, \cite{M07}).

In this section, we assume that $X$ admits a \textbf{dualizing}
\textbf{complex}. Recall that a bounded complex $\DD$ of injective
$\CO_X$-modules is called \textit{dualizing} if it  has coherent
cohomology and $\CO_X\lrt\mathbf{Hom}^\bullet_{\CO_X}(\DD, \DD)$ is
a quasi-isomorphism of $\CO_X$-modules. The study of such complexes
is an old part of algebraic geometry and goes back to the
Grothendieck's work on duality theory (see \cite[Remark
1.8]{Ne10c}). Where $X=\mathrm{Spec} R$, $\DD(X)=\mathbf{D}$ is a
dualizing complex for $R$ and for any bounded complex $\F$ of flat
$R$-modules, we have a quasi-isomorphism $\delta_\F: \F \rightarrow
\mathbf{Hom}^\bullet_R(\mathbf{D}, \mathbf{D}\otimes_R^\bullet \F)$
of complexes of $R$-modules (see \cite[Section 2]{AF97}).
\begin{lemma}\label{Flat}
Let $\F$ be a complex of flat $\CO_X$-modules. Then,
$\mathbf{Hom}_{\mathrm{qc}}^\bullet(\DD,
\DD\otimes_{\CO_X}^\bullet\F)$ is a complex of pure injective flat
$\CO_X$-modules.
\end{lemma}
\begin{proof}
We know by Proposition \ref{mes1} that
$\DD\otimes^\bullet_{\CO_X}\F$ is a complex of injective
$\CO_X$-modules. Then, by Proposition \ref{z440048971} and Theorem
\ref{mainf}, $\mathbf{Hom}_{\mathrm{qc}}^\bullet(\DD,
\DD\otimes_{\CO_X}^\bullet\F)$ is a complex of pure injective flat
$\CO_X$-modules.
\end{proof}

\begin{theorem}\label{bound1}
Let $\F$ be a bounded complex of  flat $\CO_X$-modules. Then, $
\Phi_{\F}: \F \rightarrow \mathbf{Hom}_{\mathrm{qc}}^\bullet(\DD,
\DD\otimes_{\CO_X}^\bullet\F)$ is a pure  quasi-isomorphism.
\end{theorem}
\begin{proof}
Consider the following commutative diagram of distinguished
triangles in $\D(X)$

\begin{align}\label{bound11397}
\xymatrix{\F\ar[r]\ar@{=}[d]&
\mathbf{Hom}_{\mathrm{qc}}^\bullet(\DD,
\DD\otimes_{\CO_X}^\bullet\F)\ar[r]\ar[d]_\mu&\mathrm{Cone}(\Phi_\F)\ar[r]\ar[d]_{\gamma}&\Sigma\F\ar@{=}[d]\\
\F\ar[r]&\mathbf{Hom}_{\CO_X}^\bullet(H(\DD),
\DD\otimes_{\CO_X}^\bullet\F)\ar[r]&\mathrm{Cone}(\mu\Phi_\F)\ar[r]&\Sigma\F}
\end{align}where H$(\DD)$ is the cohomology of $\DD$ ($\DD$ has coherent cohomology). By the injectivity of
$\DD\otimes_{\CO_X}^\bullet\F$, Lemma \ref{hom01325871} and Acyclic
Assembly Lemma, we deduce that
$\mathbf{Hom}_{\mathrm{qc}}^\bullet(\textmd{-},
\DD\otimes_{\CO_X}^\bullet\F)$ preserves quasi-isomorphism. Hence
$\mu$ is an isomorphism in $\D(X)$. So, $\gamma$ is also an
isomorphism in $\D(X)$. In the other hand, for any $x\in X$, we have
the following commutative diagram in $\D(\CO_x)$,

\[\xymatrix{\F_x\ar[r]\ar@{=}[d]&
\mathbf{Hom}_{\CO_x}^\bullet(\DD_x,
\DD_x\otimes_{\CO_x}^\bullet\F_x)\ar[r]\ar[d]^\cong&\mathrm{Cone}(\delta_{(\F_x)})\ar[r]\ar[d]_{\cong}&\Sigma\F_x\ar@{=}[d]\\
\F_x\ar[r]&\mathbf{Hom}_{\CO_x}^\bullet(H(\DD)_x,
\DD_x\otimes_{\CO_x}^\bullet\F_x)\ar[r]&\mathrm{Cone}(\mu\Phi_{(\F)})_x\ar[r]&\Sigma\F_x}\]where
$\mathrm{Cone}(\delta_{\F_x})$ is an acyclic complex of
$\CO_x$-modules.  This shows that $\mathrm{Cone}(\mu\Phi_\F)$ is an
acyclic complex and hence by diagram \eqref{bound11397},
$\mathrm{Cone}(\Phi_\F)$ is a bounded acyclic complex of flat
$\CO_X$-module. This follows by Lemma \ref{Flat} and Proposition
\ref{elh20} that $\mathrm{Cone}(\Phi_\F)$ is pure acyclic.
\end{proof}

\begin{theorem}\label{bound19806}
Let $\F$ be a complex of  flat $\CO_X$-modules. Then, $ \Phi_{\F}:
\F \rightarrow \mathbf{Hom}_{\mathrm{qc}}^\bullet(\DD,
\DD\otimes_{\CO_X}^\bullet\F)$ is a pure  quasi-isomorphism.
\end{theorem}
\begin{proof}
Let $\F=(\CF^n,\pa_\F^n)$ be a complex of flat  $\CO_X$-modules. For
each $k\geq 1$, let

\[\xymatrix@C-0.7pc@R-0.9pc{\F^k:\cdots\ar[r]&0\ar[r]&0\ar[r]&\CF^{-k}\ar[r]&\CF^{-k+1}
\ar[r]&\cdots\ar[r]&\CF^{k-1}\ar[r]&\CF^k\ar[r]&0\ar[r]&0\ar[r]&\cdots}\]and
$ \Phi_{\F^k}: \F^k \rightarrow
\mathbf{Hom}_{\mathrm{qc}}^\bullet(\DD,
\DD\otimes_{\CO_X}^\bullet\F^k)$. Since $\DD$ is a fixed bounded
complex then in the complex $(\mathrm{Cone}\Phi_{\F},\delta^n)$, the
$\ker\delta^n$ (resp. $\mathrm{im}\delta^n$) can be appeared in the
bounded complex  $(\mathrm{Cone}\Phi_{\F^k},\delta_k^n)$ for some
$k$. But, by Theorem \ref{bound1}, $\mathrm{Cone}\Phi_{\F^k}$ is a
pure acyclic complex of flat $\CO_X$-modules, then, we deduce that
$\mathrm{Cone}\Phi_{\F}$ is pure acyclic.
\end{proof}

We arrive now at the main theorem of this article.

\begin{theorem}\label{t6}
There is a pair of equivalences
\begin{align}\label{chart6}
\xymatrix@C+3pc@R+3pc{\KPIFX
\ar@<.75ex>[r]^{\DD\otimes^\bullet_{\CO_X}-}&
 \KIX \ar@<0.5ex>[l]^{\mathbf{Hom}_{\mathrm{qc}}^\bullet(\DD,-)} }
\end{align}
of triangulated categories.
\end{theorem}

\begin{proof}
Consider the morphism $\Phi_{\F}:\F\lrt
\mathbf{Hom}_{\mathrm{qc}}^\bullet(\DD,\DD\otimes_{\CO_X}^\bullet\F)$
of complexes of pure injective flat $\CO_X$-modules. By Theorem
\ref{bound19806}, Cone$(\Phi_{\F})$ is a pure acyclic complex of
pure injective flat $\CO_X$-modules. We conclude by \cite[Theorem
3.6]{Ho17} that Cone$(\Phi_{\F})$ is contractible and hence
$\Phi_{\F}$ is an isomorphism in $\KPIFX$.
\end{proof}

\subsection{The affine case}
In this subsection, we assume that  $X=\mathrm{Spec}R$ is an affine
scheme with a dualizing complex $\mathbf{D}$. We show that
\ref{chart6} is the infinite completion of the Grothendieck duality
theorem.

\begin{corollary}\label{t1}
Let $R$ be a noetherian ring with a dualizing complex $\mathbf{D}$.
Then, there is a pair of equivalences
\begin{align}\label{t100}
\xymatrix@C+3pc@R+2pc{\KPIFR
\ar@<0.75ex>[r]^{\mathbf{D}\otimes^\bullet_R\texttt{-}} &
 \KIR \ar@<.5ex>[l]^{{\mathrm{H}}om^\bullet_R(\mathbf{D},\texttt{-})} }
\end{align}
of triangulated  categories.
\end{corollary}
\begin{proof}
A direct consequence of Theorem \ref{t6}.
\end{proof}

Assume that  $F:\CT'\lrt\CT$ is a triangulated  functor of
triangulated categories. The \textit{essential} \textit{image} of
$F$ in $\CT$ is the full subcategory of $\CT$ formed by all objects
which are isomorphic to $F(\mathbf{T}')$ for some
$\mathbf{T}'\in\CT'$.

\begin{corollary}
Let $R$ be a noetherian ring with a dualizing complex $\mathbf{D}$.
Then, there is a pair of equivalences
\begin{align}\label{t1001}
\xymatrix@C+3pc@R+2pc{\KPR
\ar@<0.75ex>[r]^{\mathbf{D}\otimes^\bullet_R\texttt{-}} &
 \KIR. \ar@<.5ex>[l]^{\mathrm{Hom}^\bullet_R(\mathbf{D},\texttt{-})} }
\end{align}
of triangulated  categories.
\end{corollary}

\begin{proof}
By \cite[Theorem 1.2]{HS13}, $\KPIFR=\KCOFR$ is the essential image
of $\KPR$ in $\KFR$. So, in $\KFR$, any complex of projective
$R$-modules is isomorphic to a complex of pure injective flat
$R$-modules. Therefore, by  a similar argument as in the proof of
Theorem \ref{t6} we deduce the desired equivalence of triangulated
categories.

\end{proof}

\section{Pure derived categories of $\CO_X$-modules}
Recall that an $\CO_X$-module $\CF$ is called \textit{absolutely}
\textit{pure} if any monomorphism $\CF\lrt\CG$ of $\CO_X$-modules is
pure. Let $\mathrm{Abs}X$ be the class of all absolutely pure
$\CO_X$-modules, $\KABX$ be its homotopy category and $\KPABX$ be
the full subcategory of $\KABX$ consisting of all pure acyclic
complexes. Let $\DPABX=\KABX/\KPABX$ be the pure derived category of
absolutely pure $\CO_X$-modules. In this section, we show that
$$\xymatrix@C+3pc@R+3pc{\DPFX \ar[r]^{\DD\otimes^\bullet_{\CO_X}-}&
 \DPABX }$$ is an equivalence of triangulated categories where
 $\DPFX$ is the pure derived category of flat $\CO_X$-modules.

\begin{lemma}\label{t6292001331}
An $\CO_X$-module $\CF$ is absolutely pure if and only if it is  a
pure submodule of an injective $\CO_X$-module.
\end{lemma}
\begin{proof}
The proof is straightforward.
\end{proof}

Obviously, pure injective absolutely pure $\CO_X$-modules are
injective. In the next result, as a generalization of \cite[Theorem
3]{Me70}, we give a characterization of locally noetherian schemes.
\begin{theorem}\label{t62921}
The scheme $X$ is locally noetherian if and only if any absolutely
pure $\CO_X$-module is injective.
\end{theorem}
\begin{proof}
Assume that $X$ is locally noetherian and $\CF$ is an absolutely
pure $\CO_X$-module. Then, by Lemma \ref{t6292001331}, there is a
pure monomorphism $\CF\lrt\CQ$ where $\CQ$ is an  injective
$\CO_X$-module. So, by Proposition \ref{mes1}, for an affine open
subset $U$ of $X$, there is a pure monomorphism $\CF|_U\lrt\CQ|_U$
where $\CQ|_U$ is injective. Then, by Lemma \ref{t6292001331},
$\CF|_U$ is an absolutely pure $\CO_U$-module and hence by
\cite[Theorem 3]{Me70}, it is injective. Therefore, by Proposition
\ref{mes1}, $\CF$ is injective. Conversely, assume that any
absolutely pure $\CO_X$-module is injective. It is enough to prove
that if $U=\mathrm{Spec}A$  is an affine open subset of $X$, then
$A$ is a noetherian ring. By \cite[Theorem 3]{Me70},  $A$ is
noetherian if and only if any absolutely pure $A$-module is
injective. So, for a given absolutely pure $A$-module $M$, there is
a pure monomorphism $M\lrt Q$ where $Q$ is an injective $A$-module.
Since pure monomorphisms in the category $\mathfrak{Qco}U$ of all
quasi-coherent $\CO_U$-modules are directed colimits of split
monomorphisms, by \cite[Lemma B6. pp.410]{TT},
$j_*(\widetilde{M})\lrt j_*(\widetilde{Q})$ ($j:U\lrt X$ is
inclusion) is a pure monomorphism of $\CO_X$-modules where
$j_*(\widetilde{Q})$ is injective. Then, $j_*(\widetilde{M})$ is
absolutely pure and hence it is injective. So, by proposition
\ref{mes1},  $M$ is injective and the proof completes.
\end{proof}
\begin{corollary}\label{t629210}
$\KIX=\KABX$.
\end{corollary}

\begin{lemma}\label{t629}
Any pure acyclic complex of injective $\CO_X$-modules is
contractible.
\end{lemma}
\begin{proof}
Let $$\I:\cdots \rt \CI^{n-1} \st{\pa^{n-1}}{\rt} \CI^n
\st{\pa^{n}}{\rt} \CI^{n+1} \rt \cdots$$ be a pure acyclic complex
of injective $\CO_X$-modules. For an arbitrary $n\in\Z$, $\Ker\pa^n$
is a pure submodule of the injective $\CO_X$-module $\CI^n$. By
Lemma \ref{t6292001331}, it is absolutely pure and by Theorem
\ref{t62921}, it is injective. So,  $\I$ is contractible.
\end{proof}

\begin{corollary}\label{t621055}
There is an equivalence  $\KIX\cong\DPABX$ of triangulated
categories.
\end{corollary}
\begin{proof}
By Lemma \ref{t629}, $\KPABX=\K_{\mathrm{pac}}(\mathrm{Inj}X)=0$.
So, by Corollary \ref{t629210},  we conclude that
$\DPABX=\KIX/\KPIX\cong\KIX$.
\end{proof}

Let $\DPFX=\KFX/\KPFX$ be the pure derived category of flat
$\CO_X$-modules. By \cite[Proposition 3.7]{Ho17}, there is an
equivalence $\DPFX\lrt\KCOFX=\KPIFX$. If $X$ admits a dualizing
complex $\DD$, then, we have the following equivalences $$
\xymatrix@C+3pc@R+3pc{\KPIFX
\ar@<.75ex>[r]^{\DD\otimes^\bullet_{\CO_X}-}&
 \KIX \ar@<0.5ex>[l]^{\mathbf{Hom}_{\mathrm{qc}}^\bullet(\DD,-)} }$$of triangulated
categories. Note that $\K(\mathrm{Piabs}X)=\KIX$ where
$\K(\mathrm{Piabs}X)$ is the homotopy category of
 pure injective absolutely pure $\CO_X$-modules.
\begin{corollary}\label{t61}
Assume that $X$ has a dualizing complex $\DD$. Then, there is a pair
of equivalences

\begin{align}
\xymatrix@C+3pc@R+3pc{\DPFX
\ar@<.75ex>[r]^{\DD\otimes^\bullet_{\CO_X}-}&
 \DPABX \ar@<0.5ex>[l]^{\mathbf{Hom}_{\mathrm{qc}}^\bullet(\DD,-)} }
\end{align}
of triangulated categories.
\end{corollary}

\end{document}